\documentclass[12pt]{amsart}

\usepackage{enumerate,url,amssymb,  nicefrac, mathrsfs, graphicx, pdfsync, color}
\newtheorem{theorem}{Theorem}[section]

\newtheorem*{lemma*}{Lemma}

\newtheorem{thm}{Theorem}[section]
\newtheorem{lem}{Lemma}[section]
\newtheorem{prop}{Proposition}[section]

\newtheorem{defn}{Definition}[section]

\numberwithin{equation}{section}

\theoremstyle{definition}
\newtheorem{definition}[theorem]{Definition}

\theoremstyle{remark}

\numberwithin{equation}{section}

\newcommand{\abs}[1]{\lvert#1\rvert}

\newcommand{\DD}{\mathbb{D}}

\newcommand{\dd}{\mathbb{D}}

\newcommand{\onto}{\xrightarrow[]{{}_{\!\!\textnormal{onto\,\,}\!\!}}}

\DeclareMathOperator{\diam}{diam}

\def\rr{{\mathbb R}}

\def\dd{{\mathbb D}}

\def\fz{\infty}

\def\boz{{\Omega}}

\def\bint{{\ifinner\rlap{\bf\kern.25em--}
\int\else\rlap{\bf\kern.45em--}\int\fi}\ignorespaces}

\def\bbint{{\ifinner\rlap{\bf\kern.25em--}
\hspace{0.078cm}\int\else\rlap{\bf\kern.45em--}\int\fi}\ignorespaces}

\def\diam{{\mathop\mathrm{\,diam\,}}}

\def\r{\right}
\def\lf{\left}

\long\def\colred#1\endred{{\color{red}#1}}
\long\def\colgreen#1\endgreen{{\color{green}#1}}
\long\def\colmagenta#1\endmagenta{{\color{magenta}#1}}
\long\def\colblue#1\endblue{{\color{blue}#1}}
\long\def\colyellow#1\endyellow{{\color{yellow}#1}}

\begin{document}

\title{Sobolev extension on $L^{p}$-quasidisks}



\author[Z. Zhu]{Zheng Zhu}
\address{Department of Mathematics and Statistics, P.O.Box 35 (MaD) FI-40014 University of Jyv\"askyl\"a, Finland}
\email{zheng.z.zhu@jyu.fi}

\thanks{The author was supported by the Academy of Finland (project No. 323960). The author thanks Prof. P. Koskela and Prof. J. Onninen for some interesting discussion and Prof. T. Kilpel\"ainen for improving the writing of the paper.}




\maketitle

\begin{abstract} 
In this paper, we study the Sobolev extension property of $L^p$-quasidisks which are the generalizations of classical quasidisks. After that, we also find some applications of this property. 
\end{abstract}


\section{Introduction}
Let $\boz\subset\rr^2$ be a domain. A homeomorphism $h:\boz\to\rr^2$ is said to be quasiconformal, if $h\in W^{1, 2}_{\rm loc}(\boz, \rr^2)$ and the inequality 
\begin{equation}\label{eq:quasi}
|Dh(z)|^2\leq KJ_h(z),
\end{equation}
holds for almost every $z\in\boz$ with a constant $K\in[1, \fz)$ independent of $x$. A bounded simply connected domain $\boz\subset\rr^2$ is called a quasidisk, if there exists a global quasiconformal mapping from $h\colon\rr^2\onto\rr^2$ with $h(\boz)=\dd$. Quasidisks have lots of nice geometrical and potential properties, see the textbook by Gehring and Hag \cite{GeHag} and references therein. For example, a bounded simply connected planar domain is a quasidisk if and only if it is a uniform domain. Uniform domains were first introduced by Martio and Sarvas \cite{MS}. A domain $\boz\subset\rr^2$ is called uniform, if there exists a constant $C<\fz$ such that for every $z_1, z_2\in\boz$ there is a curve $\gamma_{z_1, z_2}\subset\boz$ with endpoints $z_1$ and $z_2$ and 
\[l(\gamma_{z_1, z_2})\leq Cd(z_1, z_2)\]
and for all $z\in\gamma_{z_1,z_2}$ it holds 
\[\min\left\{l(\gamma_{z_1,z}), l(\gamma_{z, z_2})\right\}\leq Cd(z, \rr^2\setminus\boz),\]
where $\gamma_{z_1, z}$ and $\gamma_{z, z_2}$ mean the subcurves of $\gamma_{z_1, z_2}$ from $z_1$ to $z$ and from $z$ to $z_2$ respectively. 
By Jone's result in \cite{Jones}, uniform domains hence quasidisks are Sobolev $(p, p)$-extension domains for arbitrary $1\leq p\leq\fz$. We say $\boz\subset\rr^2$ is a Sobolev $(p, q)$-extension domain, if for every $u\in W^{1, p}(\boz)$, there exists a function $E(u)\in W^{1, q}(\rr^2)$ with $E(u)\big|_\boz\equiv u$ and 
$$\|E(u)\|_{W^{1, q}(\rr^2)}\leq C\|u\|_{W^{1, p}(\boz)}$$
 with a positive constant $C$ independent of $u$.

There are lots of planar simply connected domains are not quasidisks, for example inward and outward cuspidal domains, see \cite{GKT, HK, IMb, IOZ, KT1, KT3, KT2}. Hence, it is natural to study generalizations of quasiconformal mappings. For instance, we concentrate on homeomorphisms of finite distortion here. A homeomorphism $h:\boz\to\rr^2$ is said to be a homeomorphism of finite distortion, if $h\in W^{1, 1}_{\rm loc}(\boz, \rr^2)$ and the inequality 
\begin{equation}\label{eq:HFD}
|Dh(z)|^2\leq K(z)J_h(z), 
\end{equation}
holds for almost every $z\in\boz$ with a measurable function $K(z)\in [1, \fz)$. For a homeomorphism of finite distortion $h$, we denote $K_h$ to be the optimal distortion function for (\ref{eq:HFD}) which will be defined below. If $K_h\in L^\fz(\boz)$, then $h$ is quasiconformal. The inverse of a quasiconformal mapping is still quasiconformal. However, if we relax the regularity of the distortion function from essentially boundedness to some other weaker one, we cannot hope the distortion of the inverse can attain the same regularity, see \cite{HKarma, HK}. Motivated by the Sobolev extension property of quasidisks, we can arise two following interesting problems.
\begin{enumerate}
\item What is the best Sobolev extension property of those domains which can be mapped onto the unit disk by homeomorphisms of finite distortion whose distortion functions satisfy some regularity?

\item What is the best Sobolev extension property of domains which are images of the unit disk under homeomorphisms of finite distortion whose distortion functions satisfy some regularity?
\end{enumerate}

In this paper, we mainly concentrate on the first problem with the condition that distortion functions are locally $L^p$-integrable for $1\leq p<\fz$. Under this condition, the corresponding bounded domains are called $L^p$-quasidisks. The terminology $L^p$-quasidisk was firstly introduced in the paper \cite{IOZ} by author with Iwaniec and Onninen, where we gave characterizations to polynomial cuspidal domains which are $L^p$-quasidisks simultaneously. In \cite{IOZ},  we showed that every bounded domain with a rectifiable boundary is a $L^1$-quasidisk. Hence, distortion function is locally $L^1$-integrable is not enough to permit Sobolev extension.
\begin{thm}\label{thm:L1}
There exists a $L^1$-quasidisk which is not a Sobolev $(p, q)$-extension domain for any $1\leq q\leq p<\fz$.
\end{thm} 
The main result in \cite{IOZ} tells us that for a fixed $1<p<\fz$, polynomial cuspidal domains with a relatively not very sharp singularities are $L^p$-quasidisks. Hence, for $1<p<\fz$, we can not hope every $L^p$-quasidisk is a Sobolev $(k, k)$-extension domain for $1\leq\ k<\fz$, see \cite{Mazya1, Mazya2, Mazya3, MP}. The following theorem will imply the Sobolev extension property of $L^p$-quasidisks which is actually a special case of the result in the theorem.
\begin{thm}\label{thm:extension}
Let $\boz\subset\rr^2$ be a bounded simply connected domain and $0<R<\fz$ be a large enough constant with $\overline\boz\subset B(0, R)$. If there exists a homeomorphism of finite distortion $h:\rr^2\onto\rr^2$ with $h(\boz)=\dd$ and $K_h\in L^p(\boz)\cap L^q(B(0, R)\setminus\overline\boz)$ for $1<p, q\leq\fz$. Then $\boz$ is a Sobolev $\lf(\frac{2p}{p-1}, \frac{2q}{q+1}\r)$-extension domain and $\rr^2\setminus\overline\boz$ is a Sobolev $\lf(\frac{2q}{q-1}, \frac{2p}{p+1}\r)$-extension domain. 
\end{thm}
Inward cuspidal domains with polynomial-type singularity will shows us the sharpness of this result. Also, the combination of this result with the Sobolev extension property of polynomial inward cuspidal domains due to Maz'ya and Poborchi \cite{Mazya1, Mazya2, Mazya3, MP} will give a new and simpler proof to the necessary part of the main result in \cite{IOZ}.  By a result due to Hencl and Koskela in \cite{HKarma}, for a homeomorphism of finite distortion $h:\rr^2\onto\rr^2$ with $\exp\lf(\lambda K_h\r)\in L^1_{\rm loc}(\rr^2)$ for some large enough positive constant $\lambda$, we have $K_{h^{-1}}\in L^p_{\rm loc}(\rr^2)$ for $1<p<c_2\lambda$ with a constant $c_2<2$ independent of $p$. Hence, as an application of Theorem \ref{thm:extension}, we have the following result towards to the second problem above. 
\begin{thm}\label{thm:exp}
Let $h:\rr^2\onto\rr^2$ be a homeomorphism of finite distortion with $\exp\lf(\lambda K_h\r)\in L^1_{\rm loc}(\rr^2)$ for some $\lambda>0$ large enough. Then $\boz:=h(\dd)$ is a Sobolev $\lf(\frac{2p_\lambda}{p_\lambda-1}, \frac{2p_\lambda}{p_\lambda+1}\r)$-extension domain for some $1<p_\lambda<\fz$. 
\end{thm}
However, we have not obtained the best Sobolev extension property for all bounded simply connected domains which are images of the unit disk under global homeomorphisms of finite distortion with locally exponentially integrable distortion function.

\section{Preliminarily}
The notation $\boz$ always means a domain in the Euclidean plane $\rr^2$. $\widehat{\rr^2}:=\rr^2\cup\{\fz\}$ is the one-point compactification of the plane $\rr^2$. $B(z, r)$ is a disk with the center $z\in\rr^2$ and radius $0<r<\fz$. $\dd:=B(0, 1)$ means the unit disk in $\rr^2$. For $0<r<R<\fz$, we denote $A(r, R):=B(0, R)\setminus\overline{B(0, r)}$ to be an annulus. Typically, $C$ will be a constant that depend on various parameters and may differ even on the same line of inequalities. For a measurable subset $A\subset\boz$ with $0<|A|<\fz$ and a function $u\in L^1_{\rm loc}(\boz)$, $u_A$ is the integral average defined by setting
\[u_A:=\bint_{A}u(z)dz=\frac{1}{|A|}\int_{A}u(z)dz.\]

Let us give the definition of Sobolev spaces first.
\begin{definition}\label{de:sobolev}
Let $1\leq p\leq\fz$ and $\boz\subset\rr^2$ be a domain. we say a function $u\in L^1_{\rm loc}(\boz)$ belongs to the homogeneous Sobolev space $\dot W^{1, p}(\boz)$ if it is weakly differentiable and its weak derivative satisfies $|\nabla u|\in L^p(\boz)$. The homogeneous Sobolev space $\dot W^{1, p}(\boz)$ is equipped with the semi-norm 
\[\|u\|_{\dot W^{1, p}(\boz)}:=\lf(\int_\boz|\nabla u(z)|^pdz\r)^{\frac{1}{p}}.\] 
If $u\in L^p(\boz)$ at the same time, we say $u$ is contained in the Sobolev space $W^{1, p}(\boz)$. The Sobolev space $W^{1, p}(\boz)$ is equipped with the norm 
\[\|u\|_{W^{1,p}(\boz)}:=\lf(\int_\boz|u(z)|^p+|\nabla u(z)|^pdz\r)^{\frac{1}{p}}.\]
 \end{definition}
Next, we define Sobolev extension domains and homogeneous Sobolev extension domains.
\begin{defn}\label{de:exten}
Let $1\leq q\leq p\leq\fz$. A bounded domain $\boz\subset\rr^2$ is said to be a Sobolev $(p, q)$-extension domain, if for every $u\in W^{1, p}(\boz)$ there exists an extension function $E(u)\in W^{1, q}(\rr^2)$ with $E(u)\big|_\boz\equiv u$ and $$\|E(u)\|_{W^{1,q}(\rr^2)}\leq C\|u\|_{W^{1, p}(\boz)}$$
 for a constant $C$ independent of $u$. Simply replace $W^{1, p}(\boz)$ by $\dot W^{1, p}(\boz)$ and replace $W^{1, q}(\rr^2)$ by $\dot W^{1, q}(\rr^2)$ in above, we get the definition of homogeneous Sobolev $(p, q)$-extension domains.
\end{defn} 


In \cite{HKjam}, authors proved that a bounded domain is homogeneous Sobolev $(p, p)$-extension if and only of it is a Sobolev $(p, p)$-extension for all $1\leq p<\fz$. Their argument at least also implies that a bounded homogeneous Sobolev $(p, q)$-extension domain is also a Sobolev $(p, q)$-extension domain for $1\leq q\leq p<\fz$. For the convenience of readers, we represent this observation here.
\begin{lem}\label{le:HSEtoSE}
Let $1\leq q\leq p<\fz$. A bounded homogeneous Sobolev $(p,q)$-extension domain is also a Sobolev $(p, q)$-extension domain.
\end{lem}
\begin{proof}
Let $\boz\subset\rr^2$ be a bounded homogeneous Sobolev $(p, q)$-extension domain for $1\leq q\leq p<\fz$. And let $B\subset\rr^2$ be a large enough disk with $\boz\subset\subset B$. Fix $u\in W^{1,p}(\boz)\subset\dot W^{1,p}(\boz)$. Since $\boz$ is a homogeneous Sobolev $(p,q)$-extension domain, we have an extension function $E(u)\in\dot W^{1,q}(\rr^2)$ with $E(u)\big|_\boz\equiv u$ and $$\|\nabla E(u)\|_{L^q(\rr^2)}\leq C\|\nabla u\|_{L^P(\boz)}$$ for a constant $C$ independent of $u$. By \cite[Lemma 4.2]{HKjam}, we have $E(u)\big|_B\in W^{1, q}(B)$. Then, Lemma $4.1$ in \cite{HKjam} implies
\begin{multline}\label{eq:poin}
\lf(\int_B|E(u)(z)-u_\boz|^qdz\r)^{\frac{1}{q}}\\
\leq 2^n\diam(B)\lf(\frac{|B|}{|\boz|}\r)^{\frac{1}{q}}\lf(\int_B|\nabla E(u)(z)|^qdz\r)^{\frac{1}{q}}\\
\leq C\lf(\int_\boz|\nabla u(z)|^pdz\r)^{\frac{1}{p}}.
\end{multline}
By (\ref{eq:poin}), the triangle inequality and the H\"older inequality, we obtain the desired inequality that
\begin{multline}
\lf(\int_B\abs{E(u)(z)}^q+\abs{\nabla E(u)(z)}^qdz\r)^{\frac{1}{q}}\\
\leq C\lf(\int_\boz\abs{\nabla u(z)}^p\r)^{\frac{1}{p}}+C\lf(\frac{|B|}{|\boz|}\r)^{\frac{1}{q}}\lf(\int_\boz\abs{u(z)}^qdz\r)^{\frac{1}{q}}\\
\leq C\lf(\int_\boz\abs{u(z)}^pdz+\abs{\nabla u(z)}^pdz\r)^{\frac{1}{p}}.
\end{multline}
\end{proof}

Let us define homeomorphisms of finite distortion and $L^p$-quasidisks.
\begin{defn}\label{de:HFD}
We say a homeomorphism $h:\boz\onto h(\boz)$ has finite distortion if $h\in W^{1,1}_{\rm loc}(\boz, \rr^2)$ and there is a measurable function $K:\boz\to[1, \fz]$ with $K(z)<\fz$ almost everywhere such that 
\[|Dh(z)|^2\leq K(x)J_h(z)\ {\rm for\ almost\ all}\ z\in\boz.\]
\end{defn}
For a homeomorphism of finite distortion, we define the optimal distortion function by setting 
\begin{equation}\label{eq:optimal}
K_h(z):=\begin{cases}
\frac{|Dh(z)|^2}{J_h(z)}, &\ {\rm for}\ z\in\{J_h>0\},\\
1, &\ {\rm for}\ z\in\{J_h=0\}.
\end{cases}
\end{equation}
\begin{defn}\label{de:quasidisk}
Let $1\leq p\leq\fz$. An bounded simply connected domain $\boz\subset\rr^2$ is said to be a $L^p$-quasidisk, if there exists a global homeomorphism of finite distortion $h:\rr^2\onto\rr^2$ with $h(\boz)=\dd$ and $K_h\in L^p_{\rm loc}(\rr^2)$.
\end{defn}
Let $\boz\subset\rr^2$ be a bounded simply connected domain. A self homeomorphism $h:\rr^2\onto\rr^2$ is called a reflection over the boundary $\partial\boz$ if $h(\boz)=\widehat{\rr^2}\setminus\overline\boz$ and $h(z)=z$ for every $z\in\partial\boz$. Base on the homeomorphism of finite distortion in Theorem \ref{thm:extension}, we can construct a reflection $\mathcal R_h:\rr^2\onto\rr^2$ over $\partial\boz$ which can induce a bounded linear extension operator from $\dot W^{1, \frac{2p}{p+1}}(\boz)$ to $\dot W^{1, \frac{2q}{q-1}}(\rr^2)$ and a bounded linear extension operator from $\dot W^{1, \frac{2q}{q+1}}(\rr^2\setminus\overline\boz)$ to $\dot W^{1, \frac{2p}{p-1}}(\rr^2)$. In general, we say a reflection $\mathcal R:\widehat{\rr^2}\onto\widehat{\rr^2}$ over $\partial\boz$ can induce a bounded linear Sobolev $(p, q)$-extension operator for $\boz$ with $1\leq q\leq p<\fz$, if there exists a bounded Lipschitz domain $U$ containing $\partial\boz$  such that, for every function $u\in\dot W^{1, p}(\boz)$, the function $v$ defined by setting $v=u$ on $U\cap\boz$ and $v=u\circ\mathcal R$ on $U\setminus\overline{\boz}$ has a representative which belongs to the Sobolev space $\dot W^{1, q}(U)$ and we have
\begin{equation}\label{uu1}
\|v\|_{\dot W^{1, q}(U)}\leq C\|u\|_{\dot W^{1,p}(\boz)}
\end{equation}
with a positive constant $C$ independent of $u$. Similarly, we say the reflection $\mathcal R$ induces a bounded linear Sobolev $(p, q)$-extension operator for $\rr^2\setminus\overline\boz$ with $1\leq q\leq p<\fz$, if for every $u\in\dot W^{1, p}(\rr^2\setminus\overline\boz)$, the function $\tilde v$ defined by setting $\tilde v=u$ on $U\setminus\overline\boz$ and $\tilde v=u\circ\mathcal R$ on $U\cap\boz$ has a representative which belongs to the Sobolev space $\dot W^{1, q}(U)$ and we have 
\[\|\tilde v\|_{\dot W^{1, q}(U)}\leq C\|u\|_{\dot W^{1, p}(\rr^2\setminus\overline\boz)}\]
with a positive constant $C$ independent of $u$. By using a suitable cut-off function, it is easy to see $\boz$ or $\rr^2\setminus\overline\boz$ is a homogeneous Sobolev $(p, q)$-extension domain. Here the introduction of the bounded open set $U$ is a convenient way to overcome the non-essential difficulty that functions in $\dot W^{1,p}(G)$ do not necessarily belong to $\dot W^{1,q}(G)$ when $1\leq q<p<\fz$ and $G$ has infinite volume. The following technical lemma justifies our terminology.
\begin{prop}\label{cut-off}
Let $\boz\subset\rr^2$ be a bounded domain with $|\partial\boz|=0$ and $\mathcal R:\widehat{\rr^2}\to\widehat{\rr^2}$ be a reflection over $\partial\boz$. If $\mathcal R$ induces a bounded linear extension operator from $\dot W^{1,p}(\boz)$ to $\dot W^{1,q}(\rr^2)$ in the sense of (\ref{uu1}) $($from $\dot W^{1,p}(\rr^2\setminus\overline\boz)$ to $\dot W^{1,q}(\rr^2)$, respectively$)$ for $1\leq q\leq p<\fz$, then $\boz$ ($\rr^2\setminus\overline\boz$, respectively) is a homogeneous Sobolev $(p, q)$-extension domain with a linear extension operator. 
\end{prop}
\begin{proof}
We only consider the case of $\boz$, since the case of $\rr^2\setminus\overline\boz$ is analogous. Let $U\subset\rr^2$ be the corresponding Lipschitz domain which contains $\partial\boz$. For a given function $u\in\dot W^{1,p}(\boz)$, we define a function $E_{\mathcal R}(u)$ by setting
\begin{equation}
E_{\mathcal R}(u)(z):=\left\{\begin{array}{ll}\label{equa:E_r(u)}
u(\mathcal R(z)),&\ {\rm for}\ z\in U\setminus\overline{\boz},\\
0,&\ {\rm for}\ z\in\partial\boz,\\
u(z),&\ {\rm for}\ z\in \Omega.
\end{array}\right.
\end{equation}
Then $E_\mathcal R(u)$ has a representative that belongs to $\dot W^{1,q}(U)$ with 
$$\|E_\mathcal R(u)\|_{\dot W^{1,q}(U)}\leq C\|u\|_{\dot W^{1,p}(\boz)}.$$
Since $U\subset\rr^2$ is a bounded Lipschitz domain, by the result due to Herron and Koskela \cite[Theorem 4.5]{HKjam}, for every $1\leq q<\fz$, there exists a bounded linear extension operator $L:\dot{W}^{1, q}(U)\to\dot{W}^{1, q}(\rr^2)$. Although Herron and Koskela only discussed the case for $1<q<\fz$, however their argument also works for $q=1$. For every function $u\in\dot W^{1,p}(\boz)$, we define an extension function $\tilde E_{\mathcal R}(u)\in\dot{W}^{1, q}(\rr^2)$ by setting 
\begin{equation}\label{equa:exglo}
\tilde E_\mathcal R(u):=L(E_\mathcal R(u)). 
\end{equation}
Then, we have $\tilde E_{\mathcal R}(u)\big|_\boz\equiv u$ with 
\[\|\tilde E_\mathcal R(u)\|_{\dot W^{1, q}(\rr^2)}\leq C\|u\|_{\dot W^{1, p}(\boz)}\]
for a constant $C$ independent of $u$.
\end{proof}

A strictly increasing function $\rho:[0, \fz)\onto[0, \fz)$ of class $C^1(0, \fz)\cap C[0, \fz)$ is called a cuspidal function if $\rho(1)=1$, $\rho'$ is increasing in $(0, \fz)$ and
\[\lim_{x\to0^+}\rho'(x)=0.\]
 The corresponding inward and outward cuspidal domains are defined by setting 
\begin{equation}\label{eq:ineward}
\boz^i_\rho:=B(0, 1)\setminus\{(x, y)\in[0, 1]\times\rr:|y|<\rho(x)\}
\end{equation}
and 
\begin{equation}\label{eq:outward}
\boz^o_\rho:= B((2, 0), \sqrt 2)\cup\{(x, y)\in(0, 1]\times\rr: |y|<\rho(x)\}
\end{equation}
respectively. See the Figure below for the exemplary cuspidal function $\rho(x)=x^{\frac{4}{3}}$.
\begin{figure}[htbp]
\centering
\includegraphics[width=1.0\textwidth]
{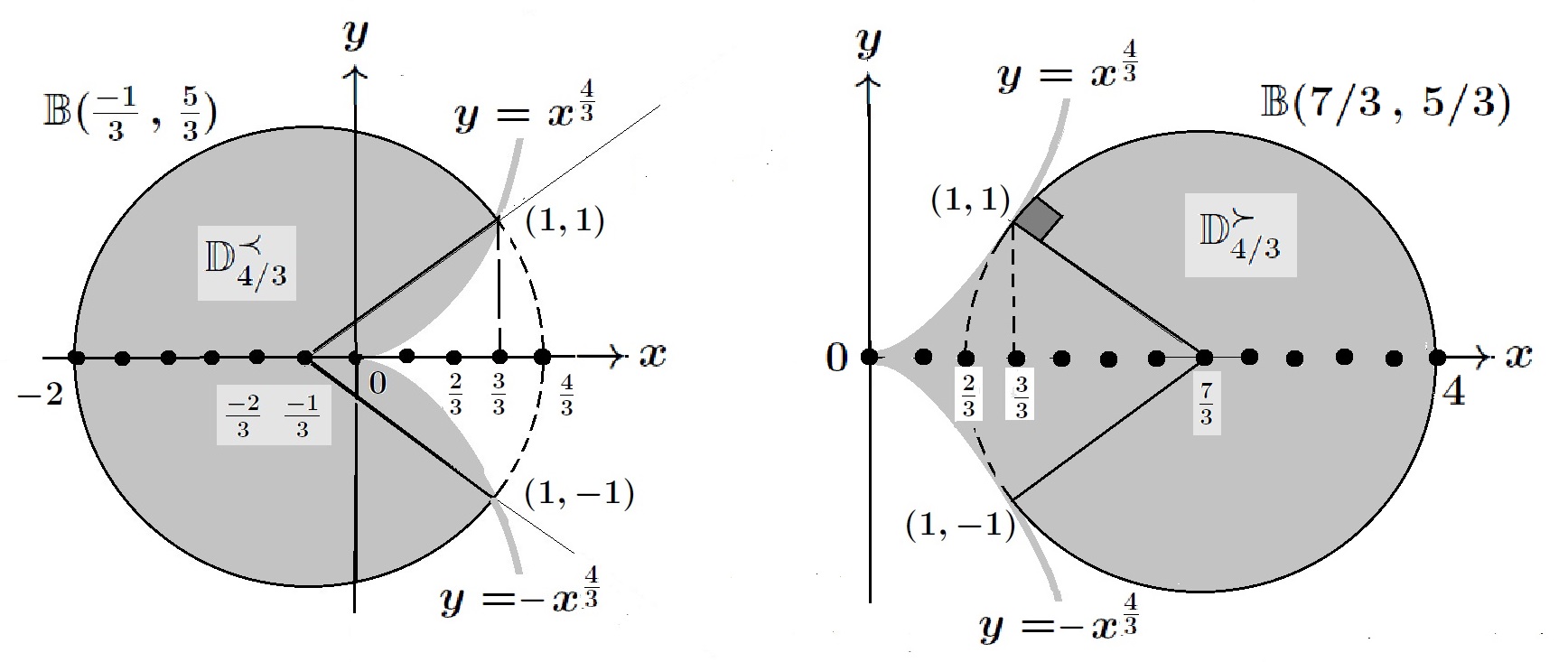}\label{cusp}
\caption{Inward and outward cuspidal domains}
\end{figure}
If the cuspidal function is $\rho(x)=x^s$ for some $1<s<\fz$, the corresponding inward and outward cuspidal domains are called polynomial inward and outward cuspidal domains with the degree $s$. The following result about Sobolev extension property of polynomial inward and outward cuspidal domains dues to Maz'ya and Poborchi \cite{Mazya1,Mazya2,Mazya3,MP}.
\begin{prop}\label{prop:extcusp}
The polynomial outward cuspidal domain $\boz^o_{t^s}$ is a Sobolev $(p, q)$-extension domain for $1\leq q<p<\fz$ if and only if $s<\frac{2p}{q}-1$. The polynomial inward cuspidal domain $\boz^i_{t^s}$ is a Sobolev $(p, q)$-extension domain for $1<q<p\leq\fz$ if and only if $s<\frac{pq+p-2q}{pq-p}$. And for every $1<s<\fz$, the polynomial inward cuspidal domain $\boz^i_{t^s}$ is a Sobolev $(1, 1)$-extension domain.
\end{prop}
By combining Theorem \ref{thm:extension} and Proposition \ref{prop:extcusp}, we can give a new and simpler proof to the necessary part of the following proposition which is the main result in the paper \cite{IOZ} by author, Iwaniec and Onninen. The proof for the sufficient part comes from the construction of desired homeomorphisms of finite distortion in \cite{IOZ}.
\begin{prop}\label{thm:main}
Let $\boz^i_{t^s}\subset\rr^2$ be a polynomial inward cuspidal domain with the degree $1<s<\fz$. Then there exists a homeomorphism of finite distortion $h:\rr^2\onto\rr^2$ with $h(\boz^i_{t^s})=\dd$ and $K_h\in L^p(\boz^i_{t^s})\cap L^q(B(0, 2)\setminus\overline{\boz^i_{t^s}})$ for $1<p,q\leq\fz$ with $\min\{p, q\}<\fz$ if and only if $s<\frac{pq+p+2q}{pq-p}$.
\end{prop}
Combine Proposition \ref{prop:extcusp} and Proposition \ref{thm:main}, we can also obtain the sharpness of Theorem \ref{thm:extension}. The polynomial inward cuspidal domain $\boz^i_{t^s}$ shows that there exists a bounded simply connected domain $\boz\subset\rr^2$ which satisfies assumptions of Theorem \ref{thm:extension}, such that we can not hope $\boz$ has a better Sobolev extension property than it is a Sobolev $\lf(\frac{2p}{p-1}, \frac{2q}{q+1}\r)$-extension domain and $\rr^2\setminus\overline\boz$ is a Sobolev $\lf(\frac{2q}{q-1}, \frac{2p}{p+1}\r)$-extension domain.


\section{Proofs of Theorem \ref{thm:L1} and Theorem \ref{thm:extension}}
\subsection{Proof of Theorem \ref{thm:L1}}
By \cite[Theorem 1.8]{IOZ}, every simply connected Jordan domain with a  rectifiable boundary is a $L^1$-quasidisk. Hence, an inward cuspidal domain $\boz^i_{\rho}$ with cuspidal function $\rho(x)=e/\exp\lf(\frac{1}{x}\r)$ is a $L^1$-quasidisk. By Proposition \ref{prop:extcusp}, $\boz^i_{\rho}$ cannot be a Sobolev $(p, q)$-extension domain for any $1\leq q\leq p<\fz$.
\subsection{Proof of Theorem \ref{thm:extension}}
Let $\boz\subset\rr^2$ be a bounded simply connected domain with $\overline\boz\subset\subset B(0, R)$ for a large enough $R>0$. Suppose there exists a homeomorphism of finite distortion $h:\rr^2\onto\rr^2$ with $h(\boz)=\dd$ and 
\begin{equation}\label{eq:integral}
\int_\boz K_h^p(z)dz+\int_{B(0, R)\setminus\overline\boz}K_h^q(z)dz<\fz.
\end{equation}

The circle inversion map $\mathcal R:\widehat{\rr^2}\onto\widehat{\rr^2}$, 
\[\mathcal R(z):=\begin{cases}
\frac{z}{|z|^2}, &\ {\rm if}\ z\neq 0,\\
\infty, &\ {\rm if}\ z=0.
\end{cases}
\]
is an anticonformal reflection over the unit circle $\partial\DD$, which means that at every point it preserves angles and reverses orientation. Then a self-homeomorphism $\widetilde{\mathcal R}:\widehat{\rr^2}\onto\widehat{\rr^2}$ defined on every $z\in\widehat{\rr^2}$ by setting $$\widetilde{\mathcal R}(z):=h^{-1}\circ\mathcal R\circ h(z)$$ is a reflection over the boundary $\partial\boz$. 

With respect to different domains $\boz$ and $\rr^2\setminus\overline{\boz}$, we divide the full proof of Theorem \ref{thm:extension} into two parts. First, let us prove the Sobolev extension property of the bounded domain $\boz$.
\begin{thm}\label{thm:boz}
Under the assumption of Theorem \ref{thm:extension}, $\boz$ is a Sobolev $\lf(\frac{2p}{p-1}, \frac{2q}{q+1}\r)$-extension domain.
\end{thm}
 \begin{proof}
By Lemma \ref{le:HSEtoSE}, it is sufficient to show that $\boz$ is a homogeneous Sobolev $\lf(\frac{2p}{p-1}, \frac{2q}{q+1}\r)$-extension domain. Let $u\in\dot W^{1, \frac{2p}{p-1}}(\boz)$ be arbitrary. We define an extension function $E_{\widetilde{\mathcal R}}(u)$ on $B(0, R)$ by setting 
\begin{equation}\label{eq:exten2}
E_{\widetilde{\mathcal R}}(u)(z):=\begin{cases}
u(\widetilde{\mathcal R}(z)), &\ {\rm if}\ z\in B(0, R)\setminus\overline\boz,\\
0, &\ {\rm if}\ z\in\partial\boz,\\
u(z), &\ {\rm if}\ z\in\boz.
\end{cases}
\end{equation}
By Proposition \ref{cut-off}, it suffices to prove $E_{\widetilde{\mathcal R}}(u)\in\dot W^{1, \frac{2q}{q+1}}(B(0, R))$ with the following inequality that
\begin{equation}\label{eq:F}
\lf(\int_{B(0,R)}|DE_{\widetilde{\mathcal R}}(u)(z)|^{\frac{2q}{q+1}}dz\r)^{\frac{q+1}{2q}}\leq C\lf(\int_{\boz}|Du(z)|^{\frac{2p}{p-1}}dz\r)^{\frac{p-1}{2p}}
\end{equation}
for a positive constant $C$ independent of $u$. Since $1<\frac{2q}{q+1}\leq2\leq\frac{2p}{p-1}<\fz$, the H\"older inequality implies
\[\lf(\int_{\boz}|DE_{\widetilde{\mathcal R}}(u)(z)|^{\frac{2q}{q+1}}dz\r)^{\frac{q+1}{2q}}\leq C\lf(\int_{\boz}|Du(z)|^{\frac{2p}{p-1}}dz\r)^{\frac{p-1}{2p}}.\] 
By \cite[Theorem 4.13]{HK}, $\partial\boz$ must be of measure zero. Hence, we only need to prove 
\begin{equation}\label{eq:G}
\lf(\int_{B(0, R)\setminus\overline\boz}|DE_{\widetilde{\mathcal R}}(u)(z)|^{\frac{2q}{q+1}}dz\r)^{\frac{q+1}{2q}}\leq C\lf(\int_{\boz}|Du(z)|^{\frac{2p}{p-1}}dz\r)^{\frac{p-1}{2p}}.
\end{equation}
We define a function $v$ on $\dd$ by setting $v(z):=u\circ h^{-1}(z)$ for every $z\in\dd$ and define an extension function $E_{\mathcal R}(v)$ on $h(B(0,R))$ with $z=(x, y)$ by setting 
\begin{equation}\label{eq:exten1}
E_{\mathcal R}(v)(z):=\begin{cases}
v(\mathcal R(z)), &\ {\rm if}\ z\in h(B(0, R))\setminus\overline\dd,\\
0, &\ {\rm if}\ z\in\partial\dd,\\
v(z), &\ {\rm if}\ z\in\dd.
\end{cases}
\end{equation}
By the definitions of functions and reflections, for every $z\in B(0, R)$, we have $$E_{\widetilde{\mathcal R}}(u)(z)=E_{\mathcal R}(v)(h(z)).$$ We divide the following argument into three steps.

\textbf{Step $1$:} We would like to show $v\in\dot W^{1, 2}(\dd)$ with
\begin{equation}\label{eq:F1}
\lf(\int_{\dd}|Dv(z)|^2dz\r)^{\frac{1}{2}}\leq C\lf(\int_{\boz}|Du(z)|^{\frac{2p}{p-1}}dz\r)^{\frac{p-1}{2p}}
\end{equation}
for a positive constant $C$ independent of $u$. By \cite[Theorem 1.6 and Theorem 2.24]{HK}, $h^{-1}\in W^{1, 1}(\dd, \rr^2)$ and it is differentiable almost everywhere on $\dd$. Hence, there exists a subset $G\subset\dd$ with $|G|=|\dd|$ and $h^{-1}$ is differentiable on every point $z\in G$. By the chain rule, for every $z\in G$, we have 
\begin{equation}\label{eq:chain}
|Dv(z)|\leq |Du(h^{-1}(z))|\cdot|Dh^{-1}(z)|.
\end{equation}
 Set $A\subset G$ to be the subset such that $J_{h^{-1}}(z)>0$ for every $z\in A$. Then by \cite[Theorem 1.6]{HK}, $|Dh^{-1}(z)|=0$ for almost every $z\in G\setminus A$. Hence, we have 
 \begin{eqnarray}\label{eq:ineq1}
 \int_{\dd}|Dv(z)|^2dz&=&\int_{G}|Dv(z)|^2dz\\
 &\leq& \int_{A}|Du(h^{-1}(z))|^2\cdot|Dh^{-1}(z)|^2dz.\nonumber
 \end{eqnarray}
 If $p=\fz$, then $h\big|_\boz$ is quasiconformal. By the fact that the inverse of a quasiconformal mapping is also quasiconformal, we have 
 \begin{multline}\label{eq:QC}
 \int_A|Du(h^{-1}(z))|^2\cdot|Dh^{-1}(z)|^2dz\\
                                  \leq\int_A|Du(h^{-1}(z))|^2J_{h^{-1}}(z)\frac{|Dh^{-1}(z)|^2}{J_{h^{-1}}(z)}dz\\
                                  \leq C\int_A|Du(h^{-1}(z))|^2J_{h^{-1}}(z)dz.
 \end{multline}
If $1<p<\fz$, the H\"older inequality implies
\begin{multline}\label{eq:holder}
\int_{A}|Du(h^{-1}(z))|^2\cdot|Dh^{-1}(z)|^2dz\\
\leq \int_{A}|Du(h^{-1}(z))|^2J_{h^{-1}}^{\frac{p-1}{p}}(z)\frac{|Dh^{-1}(z)|^2}{J^{\frac{p-1}{p}}_{h^{-1}}(z)}dz\\
\leq\lf(\int_{A}|Du(h^{-1}(z))|^{\frac{2p}{p-1}}J_{h^{-1}}(z)dz\r)^{\frac{p-1}{p}}\\
\times\lf(\int_{A}\frac{|Dh^{-1}(z)|^{2p}}{J^p_{h^{-1}}(z)}J_{h^{-1}}(z)dz\r)^{\frac{1}{p}}.
\end{multline}
For every $1<p\leq\fz$, the change of variables formula implies 
\begin{equation}\label{eq:cha1}
\int_{A}|Du(h^{-1}(z))|^{\frac{2p}{p-1}}J_{h^{-1}}(z)dz\leq \int_{h^{-1}(A)}|Du(z)|^{\frac{2p}{p-1}}dz.
\end{equation}
By \cite[Lemma A.29]{HK}, $h$ is differentiable on every $w\in h^{-1}(A)$ with $Dh(w)=(Dh^{-1}(z))^{-1}$ for $w=h^{-1}(z)$. Hence, for every $z\in A$ with $w=h^{-1}(z)$, we have 
\begin{equation}
K_{h^{-1}}(z)=\frac{|Dh^{-1}(z)|^2}{J_{h^{-1}}(z)}=\frac{|Dh(w)|^2}{J_h(w)}=K_h(w).\nonumber
\end{equation}
Hence, the change of variables formula implies 
\begin{equation}\label{eq:cha2}
\int_{A}\frac{|Dh^{-1}(z)|^{2p}}{J^p_{h^{-1}}(z)}J_{h^{-1}}(z)dz\leq \int_{\boz}K_h^p(z)dz.
\end{equation}
By combing inequalities (\ref{eq:ineq1}), (\ref{eq:QC}), (\ref{eq:holder}), (\ref{eq:cha1}) and (\ref{eq:cha2}), we obtain the desired inequality (\ref{eq:F1}).

\textbf{Step $2$:} We prove $E_{\mathcal R}(v)\in\dot W^{1, 2}(h(B(0, R)))$ with
\begin{equation}\label{eq:F2}
\int_{h(B(0, R))}|DE_\mathcal R(v)(z)|^2dz\leq C\int_{\dd}|Dv(z)|^2dz
\end{equation}
for a positive constant $C$ independent of $v$. By \cite[Lemma 4.2]{HKjam}, $v\in\dot W^{1,2}(\dd)=W^{1, 2}(\dd)$. Since $C^\fz(\rr^2)\cap W^{1, 2}(\dd)$ is dense in $W^{1, 2}(\dd)$, we can find a sequence of functions $v_k\in C^\fz(\rr^2)\cap W^{1, 2}(\dd)$ with $$\lim_{k\to\fz}v_k(z)=v(z)$$ for almost every $z\in\dd$, and 
\begin{equation}\label{eq:eq2}
\lim_{k\to\fz}\|v_k-v\|_{\dot W^{1, 2}(\dd)}=0.
\end{equation}
 With an extra diagonal argument if necessary, we can also assume the sequence of weak gradients $\{Dv_k\}$ converges to $Dv$ almost everywhere on $\dd$. For every $v_k\in C^\fz(\rr^2)\cap W^{1, 2}(\dd)$, we define the extension function $E_\mathcal R(v_k)$ as in (\ref{eq:exten1}). Since $v_k\in C^\fz(\rr^2)\cap W^{1, 2}(\dd)$, $E_\mathcal R(v_k)$ is $ACL$ on $h(B(0, R))$. By definition (\ref{eq:exten1}) and the fact that $\mathcal R$ is anticonformal, it is easy to see that we have 
\begin{equation}\label{eq:eq1}
\int_{h(B(0, R))}|DE_\mathcal R(v_k)(z)|^2dz\leq C\int_{\dd}|Dv_k(z)|^2dz,
\end{equation}
for a positive constant $C$ independent of $k$. Since $v_k$ converges to $v$ almost everywhere on $\dd$, by (\ref{eq:exten1}), $E_\mathcal R(v_k)$ converges to $E_\mathcal R(v)$ almost everywhere on $h(B(0, R))$. Since $\{Dv_k\}$ converges to $Dv$ almost everywhere on $\dd$, by the definitions of $E_\mathcal R(v_k)$ and $E_\mathcal R(v)$ in (\ref{eq:exten1}), $\{DE_\mathcal R(v_k)\}$ converges to $DE_\mathcal R(v)$ almost everywhere on $\rr^2$. By (\ref{eq:eq2}) and (\ref{eq:eq1}), we have 
\begin{equation}\label{eq:eq3}
\lim_{k\to\fz}\int_{h(B(0, R))}|DE_\mathcal R(v_k)(z)-DE_\mathcal R(v)(z)|^2dz=0.
\end{equation}
Then by Theorem $1$ in \cite[Section 1.1.13]{Mazya:book}, $E_\mathcal R(v)\in\dot W^{1, 2}(h(B(0, R)))$ with the weak gradient $DE_\mathcal R(v)$. Combine (\ref{eq:eq2}), (\ref{eq:eq1}) and (\ref{eq:eq3}), we obtain 
\begin{equation}\label{eq:eq4}
\int_{h(B(0, R))}|DE_\mathcal R(v)(z)|^2dz\leq C\int_{\dd}|Dv(z)|^2dz,
\end{equation}
for a positive constant independent of $u$.  

\textbf{Step $3$:} We prove $E_{\widetilde{\mathcal R}}(u)\in\dot W^{1, \frac{2q}{q+1}}(B(0,R))$ and we have 
\begin{multline}\label{eq:F3}
\lf(\int_{B(0, R)\setminus\overline\boz}|DE_{\widetilde{\mathcal R}}(u)(z)|^{\frac{2q}{q+1}}dz\r)^{\frac{q+1}{2q}}\\
\leq C\lf(\int_{h(B(0, R))}|DE_\mathcal R(v)(z)|^2dz\r)^{\frac{1}{2}}
\end{multline}
for a positive constant $C$ independent of $u$. As we know,  $E_{\widetilde{\mathcal R}}(u)(z)=E_\mathcal R(v)(h(z))$ for every $z\in B(0, R)$. By \cite[Theorem 1.7 and Theorem 2.24]{HK}, $h$ is differentiable almost everywhere on $B(0, R)\setminus\overline\boz$ with positive determinant value. Hence, there exists a subset $\widehat G\subset B(0, R)\setminus\overline\boz$ with $|\widehat G|=|B(0, R)\setminus\overline\boz|$ and $h$ is differentiable on every point in $z\in\widehat G$ with $J_h(z)>0$. By the chain rule, for every $z\in\widehat G$, we have 
\begin{equation}\label{eq:chain2}
|DE_{\widetilde{\mathcal R}}(u)(z)|\leq |DE_\mathcal R(v(h(z)))|\cdot|Dh(z)|.
\end{equation}
If $q=\fz$, then $h$ restricts to $B(0, R)\setminus\overline\boz$ is quasiconformal. Then we have
\begin{multline}\label{eq:QC1}
\int_{B(0, R)\setminus\overline\boz}|DE_{\widetilde{\mathcal R}}(u)(z)|^2dz\\
\leq\int_{B(0, R)\setminus\overline\boz}|DE_{\mathcal R}(v)(h(z))|^2J_h(z)\cdot\frac{|Dh(z)|^2}{J_h(z)}dz\\
\leq C\int_{B(0, R)\setminus\overline\boz}|DE_{\mathcal R}(v)(h(z))|^2J_h(z)dz.
\end{multline}
If $1<q<\fz$, then we have 
\begin{multline}\label{eq1}
\int_{B(0, R)\setminus\overline\boz}|DE_{\widetilde{\mathcal R}}(u)(z)|^{\frac{2q}{q+1}}dz=\int_{\widehat G}|DE_{\widetilde{\mathcal R}}(u)(z)|^{\frac{2q}{q+1}}dz\\
   \leq\int_{\widehat G}|DE_{\mathcal R}(v)(h(z))|^{\frac{2q}{q+1}}\cdot|Dh(z)|^{\frac{2q}{q+1}}dz.
\end{multline}

 The H\"older inequality implies 
\begin{multline}\label{eq2}
\int_{\widehat G}|DE_\mathcal R(v)(h(z))|^{\frac{2q}{q+1}}\cdot|Dh(z)|^{\frac{2q}{q+1}}dz\\
\leq\int_{\widehat G}|DE_\mathcal R(v)(h(z))|^{\frac{2q}{q+1}}J_h^{\frac{q}{q+1}}(z)\frac{|Dh(z)|^{\frac{2q}{q+1}}}{J_h^{\frac{q}{q+1}}(z)}dz\\
\leq\lf(\int_{\widehat G}|DE_\mathcal R(v)(h(z))|^2J_h(z)dz\r)^{\frac{q}{q+1}}\\
\times\lf(\int_{\widehat G}\lf(\frac{|Dh(z)|^2}{J_h(z)}\r)^qdz\r)^{\frac{1}{q}}.
\end{multline}
The change of variables formula implies 
\begin{equation}\label{eq3}
\int_{\widehat G}|DE_\mathcal R(v)(h(z))|^2J_h(z)dz\leq \int_{h(B(0, R))}|DE_\mathcal R(v)(z)|^2dz.
\end{equation}
By the definition of distortion function and the integral condition (\ref{eq:integral}), we have 
\begin{equation}\label{eq4}
\int_{\widehat G}\lf(\frac{|Dh(z)|^2}{J_h(z)}\r)^qdz\leq\int_{B(0, R)\setminus\overline\boz}K_h^q(z)dz<\fz.
\end{equation}
Combine inequalities (\ref{eq:QC1})-(\ref{eq4}), we obtain the desired inequality (\ref{eq:F3}).

Finally, combine inequalities (\ref{eq:F1}), (\ref{eq:F2}) and (\ref{eq:F3}), we obtain the desired inequality (\ref{eq:G}) and finish the proof of Theorem \ref{thm:extension}.
\end{proof}

Next, we consider the Sobolev extension property of the complementary domain $\rr^2\setminus\overline\boz$. 
\begin{thm}\label{thm:out}
Under the assumption of Theorem \ref{thm:extension}, $\rr^2\setminus\overline\boz$ is a Sobolev $\lf(\frac{2q}{q-1}, \frac{2p}{p+1}\r)$-extension domain.
\end{thm}
The idea of the proof is very similar with the proof of last theorem. The main difference is that $\rr^2\setminus\overline\boz$ is unbounded here. Hence, we give a sketch proof here.
\begin{proof}
From the geometrical viewpoint, it is easy to see $\rr^2\setminus\overline\boz$ has a same Sobolev extension property with the bounded domain $B(0, R)\setminus\overline\boz$. Hence, by Lemma \ref{le:HSEtoSE}, it suffices to prove $B(0, R)\setminus\overline\boz$ is a homogeneous Sobolev $\lf(\frac{2q}{q-1}, \frac{2p}{p+1}\r)$-extension domain. Since $h:\rr^2\onto\rr^2$ is a homeomorphism with $h(\boz)=\dd$, there exist two positive constants $\epsilon_1>0$ and $\epsilon_2>0$ such that $B(0, 1+\epsilon_1)\subset h(B(0, R))$ and $$\mathcal R(A(1-\epsilon_2,1+\epsilon_1))=A(1-\epsilon_2, 1+\epsilon_1).$$ We define $\mathcal A:=h^{-1}(A(1-\epsilon_2, 1))$, $\mathbf U:=\lf(B(0, R)\setminus\boz\r)\cup\mathcal A$ and $\widetilde{\mathbf U}:=h(\mathbf U)$. Let $u\in\dot W^{1, \frac{2q}{q-1}}(\rr^2\setminus\overline\boz)$ be arbitrary. Then $$u\big|_{B(0, R)\setminus\overline\boz}\in\dot W^{1, \frac{2q}{q-1}}(B(0, R)\setminus\overline\boz).$$ To simplify the notation, we still denote $u\big|_{B(0, R)\setminus\overline{\boz}}$ by $u$. Then we define a function $E_{\widetilde{\mathcal R}}(u)$ on $\mathbf U$ by setting 
\begin{equation}\label{eq:Exten1}
E_{\widetilde{\mathcal R}}(u)(z):=\begin{cases}
u(\widetilde{\mathcal R}(z)), &\ {\rm if}\ z\in\mathcal A,\\
0, & \ {\rm if}\ z\in\partial\boz,\\
u(z), & \ {\rm if}\ z\in B(0, R)\setminus\overline{\boz}.
\end{cases}
\end{equation}
It suffices to prove $E_{\widetilde{\mathcal R}}(u)\in\dot W^{1, \frac{2p}{p+1}}(\mathbf U)$ with 
\begin{multline}\label{eq:E1}
\lf(\int_{\mathbf U}|DE_{\widetilde{\mathcal R}}(u)(z)|^{\frac{2p}{p+1}}dz\r)^{\frac{p+1}{2p}}\\
\leq C\lf(\int_{B(0, R)\setminus\overline\boz}|Du(z)|^{\frac{2q}{q-1}}dz\r)^{\frac{q-1}{2q}}.
\end{multline}
for a positive constant $C$ independent of $u$. By the definition of $E_{\widetilde{\mathcal R}}(u)$, the H\"older inequality implies
\begin{multline}
\lf(\int_{B(0, R)\setminus\overline\boz}|DE_{\widetilde{\mathcal R}}(u)(z)|^{\frac{2p}{p+1}}\r)^{\frac{p+1}{2p}}\\
\leq C\lf(\int_{B(0, R)\setminus\overline\boz}|Du(z)|^{\frac{2q}{q-1}}dz\r)^{\frac{q-1}{2q}}
\end{multline}
for a positive constant $C$ independent of $u$. By \cite[Theorem 4.13]{HK}, $\partial\boz$ must be of measure zero. Hence, it suffices to prove the inequality 
\begin{multline}\label{eq:E11}
\lf(\int_{\mathcal A}|DE_{\widetilde{\mathcal R}}(u)(z)|^{\frac{2p}{p+1}}dz\r)^{\frac{p+1}{2p}}\\
\leq C\lf(\int_{B(0, R)\setminus\overline\boz}|Du(z)|^\frac{2q}{q-1}dz\r)^{\frac{q-1}{2q}}
\end{multline}
for a positive constant $C$ independent of $u$. We define a function $v\in h(B(0, R)\setminus\overline{\boz})$ by setting $v(z):=u(h^{-1}(z))$ for every $z\in h(B(0, R)\setminus\overline\boz)$. By a similar argument to the inequality (\ref{eq:F1}), we obtain $v\in\dot W^{1, 2}(h(B(0, R)\setminus\overline\boz))$ with
\begin{equation}\label{eq:F11}
\lf(\int_{h(B(0, R)\setminus\overline\boz)}|Dv(z)|^2dz\r)^{\frac{1}{2}}\leqq C\lf(\int_{B(0,R)\setminus\overline\boz}|Du(z)|^{\frac{2q}{q-1}}\r)^{\frac{q-1}{2q}}
\end{equation}
for a positive constant $C$ independent of $u$. Define an extension function $E_{\mathcal R}(v)$ on $\widetilde{\mathbf U}$ by setting
\begin{equation}\label{eq:Exten2}
E_{\mathcal R}(v)(z):=\begin{cases}
v(\mathcal R(z)), &\ {\rm if}\ z\in A(1-\epsilon_2, 1),\\
0, &\ {\rm if}\ z\in\partial\dd,\\
v(z), &\ {\rm if}\ z\in\widetilde{\mathbf U}\setminus\overline{\dd}.
\end{cases}
\end{equation}
It is easy to see for every $z\in\mathbf U$, we have 
\[E_{\widetilde{\mathcal R}}(u)(z)=E_{\mathcal R}(v)(h(z)).\]
 By (\ref{eq:F11}), $v\big|_{A(1, 1+\epsilon_1)}$ belongs to the class $\dot W^{1, 2}(A(1, 1+\epsilon_1))$. To simplify the notation, we still denote $v\big|_{A(1, 1+\epsilon_1)}$ by $v$. By facts that $\mathcal R(A(1-\epsilon_2, 1+\epsilon_1))=A(1-\epsilon_2, 1+\epsilon_1)$, $\dot W^{1, 2}(A(1-\epsilon_2, 1+\epsilon_1))=W^{1, 2}(A(1-\epsilon_2, 1+\epsilon_1))$ and $C^\fz(\rr^2)\cap W^{1, 2}(A(1-\epsilon_2, 1+\epsilon_1))$ is dense in $W^{1, 2}(A(1-\epsilon_2, 1+\epsilon_1))$. With a similar argument to the inequality (\ref{eq:eq4}), we obtain $E_\mathcal R(v)\in\dot W^{1, 2}(\widetilde{\mathbf U})$ with 
 \begin{multline}\label{eq:EQ4}
 \int_{A(1-\epsilon_2, 1+\epsilon_1)}|DE_{\mathcal R}(v)(z)|^2dz\\
 \leq C\int_{A(1, 1+\epsilon_1)}|Dv(z)|^2dz\\
 \leq C\int_{h(B(0, R)\setminus\overline\boz)}|Dv(z)|^2dz
 \end{multline}
 for a positive constant $C$ independent of $u$. Then, with a similar argument to the inequality (\ref{eq:F3}), we obtain the inequality 
 \begin{multline}\label{eq:F33}
 \lf(\int_{\mathcal A}|DE_{\mathcal R}(u)(z)|^{\frac{2p}{p+1}}dz\r)^{\frac{p+1}{2p}}\\
  \leq C\lf(\int_{A(1-\epsilon_2, 1)}|DE_\mathcal R(v)(z)|^2dz\r)^{\frac{1}{2}}
 \end{multline}
 for a positive constant $C$ independent of $u$. 
 
 Finally, by combining inequalities (\ref{eq:F11}), (\ref{eq:EQ4}) and (\ref{eq:F33}), we obtain the desired inequality (\ref{eq:E11}).
\end{proof}

\section{Sharpness and applications of Theorem \ref{thm:extension}} 
\subsection{Sharpness of Theorem \ref{thm:extension}}
In this section, we discuss the sharpness of Theorem \ref{thm:extension}. Let $\boz=\boz^i_{t^s}\subset\rr^2$ be a polynomial inward cuspidal domain with the degree $1<s<\fz$ such that there exists a homeomorphism of finite distortion $h:\rr^2\onto\rr^2$ with $h(\boz)=\dd$, and the distortion function satisfies the integral condition (\ref{eq:integral}). It is easy to see $\rr^2\setminus\overline{\boz^i_{t^s}}$ has a same Sobolev extension property with the polynomial outward cuspidal domain $\boz^o_{t^s}$. By combining Proposition \ref{prop:extcusp} and Proposition \ref{thm:main}, we obtain that we cannot obtain a better Sobolev extension result than $\boz$ is a Sobolev $\lf(\frac{2p}{p-1}, \frac{2q}{q+1}\r)$-extension domain and $\rr^2\setminus\overline{\boz^i_{t^s}}$ is a Sobolev $\lf(\frac{2q}{q-1}, \frac{2p}{p+1}\r)$-extension domain. It shows the sharpness of Theorem \ref{thm:extension}.

\subsection{A new proof of the necessity of Proposition \ref{thm:main}} Actually, by combing Theorem \ref{thm:extension} and Proposition \ref{prop:extcusp}, we give a new proof to the necessary part of Proposition \ref{thm:main}, which is entirely different with the proof in paper \cite{IOZ} where we gave an argument based on some geometrical facts.

\subsection{An application of Theorem \ref{thm:extension}}
In this section, we prove Theorem \ref{thm:exp}, which can be regarded as an application of Theorem \ref{thm:extension}.
\begin{proof}[Proof of Theorem \ref{thm:exp}]
By \cite[Theorem 1.6]{HKarma}, there exists a positive constant $0<\beta<2$ such that for every homeomorphism of finite distortion $h:\rr^2\onto\rr^2$ with $\exp\lf(\lambda K_h\r)\in L^1_{\rm loc}(\rr^2)$ with $\lambda>0$, we have $K_{h^{-1}}\in L^{p_\lambda}_{\rm loc}(\rr^2)$ with $p_\lambda=\beta\lambda$. Let $h:\rr^2\onto\rr^2$ be a homeomorphism of finite distortion with $\exp\lf(\lambda K_h\r)\in L^1_{\rm loc}(\rr^2)$ for some large enough $\lambda>0$ with $p_\lambda=\beta\lambda>1$. Then by Theorem \ref{thm:extension}, $\boz:=h^{-1}(\dd)$ is a $L^{p_\lambda}$-quasidisk and is a Sobolev $\lf(\frac{2p_\lambda}{p_\lambda-1}, \frac{2p_\lambda}{p_\lambda+1}\r)$-extension domain.
\end{proof}

\end{document}